\documentclass[11pt]{amsart}
\usepackage{tikz}
\usepackage{tikz-cd}
\usepackage{mathrsfs}
\usepackage[mathcal]{euscript}
\usepackage{graphicx}
\usepackage{amsthm,amscd}
\usepackage{calc}
\usepackage{color}

\usepackage{mathrsfs,dsfont}
\usepackage{fancyhdr,amscd}
\usepackage{anysize}
\usepackage{amsmath,amssymb,amsfonts}
\usepackage{epsfig,enumerate}
\usepackage{latexsym}
\usepackage{indentfirst, latexsym}
\usepackage{graphics}
\usepackage[all,poly,knot]{xy}
\usepackage[pagebackref]{hyperref}
\usepackage{lineno}

\newcommand{\comment}[1]{}

\usepackage{fancyhdr}
\providecommand{\U}[1]{\protect\rule{.1in}{.1in}}

\numberwithin{equation}{section}

\theoremstyle{plain}
\newtheorem{theorem}{Theorem}[section]  
\newtheorem{lemma}[theorem]{Lemma}      

\newtheorem{proposition}[theorem]{Proposition}
\newtheorem{corollary}[theorem]{Corollary}

\theoremstyle{definition}

\newtheorem{remark}[theorem]{Remark}
\newtheorem{claim}[theorem]{Claim}

\newcounter{proofstep}
\newenvironment{step}[1][]{%
  \refstepcounter{proofstep}%
  \par\vspace{12pt}%
  \noindent\textbf{Step \theproofstep. #1}\quad
}{\par\vspace{12pt}}

\begin{document}


\title[K\"ahler structure  of the total space near a K\"ahler fiber]{K\"ahler structure  of the total space near a K\"ahler fiber}

\author[Jian Chen]{Jian Chen}

\address{Jian Chen, School of Mathematics and Statistics, Central China Normal University,
		Wuhan 430079, People's Republic of China}
\email{jian-chen@whu.edu.cn}

\date{\today}
\subjclass[2020]{Primary 14B12; Secondary 32J27, 14D05, 14D06, 14A10.}
\keywords{Local deformation theory; Compact K\"ahler manifolds, Structure of families, Fibrations,
Varieties and morphisms.}

\begin{abstract}
Motivated by the Kodaira–-Spencer local stability theorem for  K\"ahler structures and by C. Li's study of K\"ahler structures on holomorphic submersions between compact complex manifolds, we establish an equivalent characterization of the K\"ahlerness of the total space near a K\"ahler fiber, in an optimal manner. The proof combines a  $(1,1)$-type lifting argument for flat sections of a local system,  observations on torsion-freeness   and a K\"ahler neighborhood criterion for compact K\"ahler submanifolds.

\end{abstract}

\maketitle
	
\setcounter{tocdepth}{1}
\tableofcontents

\section{Introduction}

As is well known, the K\"ahler stability theorem of Kodaira and Spencer \cite{KS60} asserts that, in a proper holomorphic submersion between complex manifolds, the fibers sufficiently close to a K\"ahler fiber are again K\"ahler. 
However, Proposition \ref{prop-k3} shows that this fiberwise stability property does not extend to the total space: even after shrinking the base around the given base point (over which the fiber is K\"ahler), the restricted total space may fail to be K\"ahler.

Recently, Li \cite{L25} established an elegant equivalent characterization of the K\"ahlerness of the total space for a  holomorphic submersion between compact complex manifolds, which generalizes Blanchard’s criterion for a special class of
isotrivial holomorphic submersions,  
using the Leray filtration and Deligne’s Hodge theory for the polarized variations of Hodge structures.

\begin{theorem}[{\cite[Theorem 1.2]{L25}}]\label{intro-Li-thm}
Let $X, B$ be compact complex manifolds. Assume that there is a holomorphic submersion $\pi: X \to B$.  Then there is a K\"ahler metric on $X$ iff the following conditions are all satisfied:
\begin{enumerate}[\rm{(}1\rm{)}]
\item\label{intro-li-1}
There is an element $[\omega] \in H^0\left(B, R^2 \pi_* \mathbb{R}\right)$ restricts to be a K\"ahler class on $X_b$ for any $b \in B$.
\item\label{intro-li-2}
$d_2[\omega]=0$ in $E_2^{2,1}=H^2\left(B, R^1 \pi_* \mathbb{R}\right)$.
\item\label{intro-li-3} 
$B$ is K\"ahler.
\end{enumerate}
\end{theorem}

Note that the compactness of the total space plays an important role in the proof of Theorem \ref{intro-Li-thm}.  In view of this, together with the possible failure of the K\"ahlerness of the total space  near a  K\"ahler fiber, a natural question arises: under what additional conditions can one obtain the K\"ahlerness of the total space (after shrinking the base) near a K\"ahler fiber?

The following equivalent characterization  of the K\"ahlerness of the total space near a K\"ahler fiber gives an answer to this question with optimal assumptions (see Proposition \ref{prop-k3} and Remark \ref{rem-product-k3}).

\begin{theorem}[{= Corollary \ref{cor-local-kahler}}]\label{intro-cor-local-kahler}
Let $\pi:X\to S$ be a proper holomorphic submersion from a complex
manifold $X$ onto a polydisc $S\subset\mathbb C^m$  centered at
$0$. Set $X_t:=\pi^{-1}(t)$ for each $t\in S$. Assume that
$X_0$ is K\"ahler.
Then the following conditions are equivalent:
\begin{enumerate}[\rm{(}1\rm{)}]
\item\label{intro-cond-lk-total}
There exists a polydisc $S_1\Subset S$ centered at $0$ such that
\[
X_{S_1}:=\pi^{-1}(S_1)
\]
is K\"ahler.

\item\label{intro-cond-lk-section}
There exist a polydisc $S_2\Subset S$ centered at $0$ and a section
\[
\sigma\in H^0(S_2,(R^2\pi_*\mathbb R)|_{S_2})
\]
such that, under the canonical identification
\[
(R^2\pi_*\mathbb R)_t\cong H^2(X_t,\mathbb R),
\]
one has
$\sigma_0=[\omega_0]$
for some K\"ahler form $\omega_0$ on $X_0$, and there exists a subset
$U\subset S_2$, dense in the analytic Zariski topology of $S_2$, such that, for
each $t\in U$, the class $\sigma_t\in H^2(X_t,\mathbb R)$
is represented by a smooth real $d$-closed $(1,1)$-form on $X_t$.
\end{enumerate}
\end{theorem}

We now briefly discuss the optimality of condition \eqref{intro-cond-lk-section} in Theorem \ref{intro-cor-local-kahler}.
In view of the Kodaira--Spencer stability theorem for K\"ahler structures, one might expect that condition \eqref{intro-cond-lk-section} in Theorem \ref{intro-cor-local-kahler} could be weakened, 
for instance by requiring the type-$(1,1)$ restriction condition only along a proper analytic subset of the base, while assuming that $\sigma_0$ is a K\"ahler class. However, Proposition \ref{prop-k3} and Remark \ref{rem-product-k3} show that condition \eqref{intro-cond-lk-section} in Theorem \ref{intro-cor-local-kahler} is optimal: even if the type-$(1,1)$ restriction condition is imposed along a divisor in the base, the total space may fail to be K\"ahler near a K\"ahler fiber.

We now explain the strategy of the proof of Theorem \ref{intro-cor-local-kahler}.

As in \cite{L25}, the main difficulty in producing a K\"ahler metric on the total space is to construct a smooth real $d$-closed $(1,1)$-form whose restriction to each fiber represents the prescribed value of a given flat section. In the compact setting of \cite{L25}, the author starts from a globally defined  $(1,1)$-form representing the prescribed fiberwise K\"ahler classes and then corrects this form to be closed step by step along the Leray filtration. This argument uses the compactness assumptions in an essential way (e.g., \cite[Section 3, Step 5]{L25}). In the local situation considered here, the total space  is not compact, so this  argument does not seem to be readily adaptable to our setting.

Motivated by the role of the type-$(1,1)$ restriction condition in the construction of holomorphic line bundles with prescribed  Chern class (e.g., \cite[Theorem 3.2]{C25}), our method here is to use a type-$(1,1)$ restriction  condition to control a Dolbeault obstruction. More precisely, for a flat section of $R^2\pi_*\mathbb R$, the polydisc assumption gives a topological lift to a real deRham class; choosing a smooth real closed representative, we take the Dolbeault class of its $(0,2)$-component and view it as a section of $R^2\pi_*\mathcal O_X$. The type-$(1,1)$ condition on a Zariski dense set of fibers forces the fiber values of this obstruction section to vanish on a dense set. Under natural conditions, this implies the global vanishing of the obstruction and hence yields a smooth real $d$-closed $(1,1)$-form representing the prescribed flat section fiberwisely (see the proof of Theorem \ref{thm-H11-lift}).

 As a direct corollary of Theorem \ref{thm-H11-lift}, we obtain the following equivalent characterization of the liftability of a flat section to a global smooth real $d$-closed $(1,1)$-form.

\begin{theorem}[{= Corollary \ref{cor-H11-tf}}]\label{intro-cor-H11-tf}
Let $\pi:X\to S$ be a proper holomorphic submersion onto a polydisc
$S\subset \mathbb C^m$ from a complex manifold $X$, with fibers denoted by $X_t:=\pi^{-1}(t)$.
Let $\sigma\in H^0(S,R^2\pi_*\mathbb R)$
be a section.   Denote by $\sigma_t\in H^2(X_t,\mathbb R)$
the image of the germ  $\sigma_t$  under the canonical identification
$(R^2\pi_*\mathbb R)_t\cong H^2(X_t,\mathbb R).$ 
Assume that $E:=R^2\pi_*\mathcal O_X$  is torsion-free.  

 Then the following conditions are equivalent:
\begin{enumerate}[\rm{(}1\rm{)}]
\item\label{intro-cond-tf-dense}
There exists a subset $U\subset S$, dense in $S$ in the analytic Zariski topology,
such that, for each $t\in U$, the class
$\sigma_t\in H^2(X_t,\mathbb{R})$ is represented by a smooth real
$d$-closed $(1,1)$-form on $X_t$.

\item\label{intro-cond-tf-global}
There exists a global smooth real $d$-closed $(1,1)$-form
$\widetilde{\alpha}$ on $X$ such that
\[
[\widetilde{\alpha}|_{X_t}]=\sigma_t
\qquad\text{for each }t\in S.
\]
\end{enumerate}
\end{theorem}

Using a smooth squared-distance-type function whose Levi form is semipositive along the submanifold and strictly positive in the normal directions, we establish in Theorem \ref{thm-pair-kahler} a criterion for the existence of a K\"ahler neighborhood of a compact K\"ahler submanifold.  Then Theorem \ref{intro-cor-local-kahler} follows from  Theorems \ref{intro-cor-H11-tf} and  \ref{thm-pair-kahler}.

The remaining sections are organized as follows. In Section 2, we give  some preliminary observations on torsion-free sheaves. In Section 3, we prove the type-$(1,1)$ lifting result for flat sections of $R^2\pi_*\mathbb R$. In Section 4, we establish a K\"ahler neighborhood criterion for compact K\"ahler submanifolds in a general complex manifold and then prove Theorem \ref{intro-cor-local-kahler}. In Section 5, we construct examples showing the sharpness of the type-$(1,1)$ restriction condition in Theorem \ref{intro-cor-local-kahler}. Finally, in Appendix A, we prove several basic facts on Leray spectral sequences and their compatibility with restriction and base-change morphisms.

\section{Preliminaries}\label{sec-preli}

Unless otherwise stated, throughout this paper: we use the notation $\mathbb{R}_{\bullet}$ to denote the $\mathbb{R}$-valued constant sheaf on a topological space $\bullet$. When no confusion arises, we also write $\mathbb{R}$ to mean $\mathbb{R}_{\bullet}$; for a proper morphism  $\pi: X\to S$,  for a section $\sigma\in H^0(S,R^2\pi_*\mathbb R)$ and a point $t\in S$, we use $\sigma_t$ and $\sigma(t)$ interchangeably to denote the image of the germ of $\sigma$ at $t$ under the canonical identification
\[
(R^2\pi_*\mathbb R)_t\cong H^2(X_t,\mathbb R);
\]
we use the phrase ``$A$ is dense in the analytic Zariski topology of $B$" to mean that the closure of $A$ in the analytic Zariski topology of $B$ is equal to $B$. If $B$ is irreducible, this is equivalent to saying that $A$ is not contained in any proper analytic subset of $B$. By contrast, we use the phrase ``$A$ is a dense analytic Zariski open subset of $B$" in its standard sense: $A$ is an analytic Zariski open subset of $B$ whose complement is a thin analytic subset of $B$.

Note that in general, the  zero (value) locus for a section of a coherent sheaf $\mathcal E$ need not be an analytic subset
of $S$, even when $\mathcal E$ is torsion-free.  However, the  zero locus of a section of a torsion-free sheaf is contained in a proper analytic subset, as shown by the following observation.

\begin{lemma}\label{lem-zero-locus}
Let $S$ be a connected complex manifold, and let $\mathcal E$ be a torsion-free
coherent sheaf on $S$. Let $\sigma\in H^0(S,\mathcal E)$ be a nonzero section.
Then  the zero locus
\[
Z(\sigma):=
\{t\in S\mid \sigma(t)=0\text{ in }\mathcal E_t\otimes_{\mathcal O_{S,t}}k(t)\}
\]
 of $\sigma$ is contained in a proper analytic subset of $S$, where $k(t):=\mathcal{O}_{S, t} / \mathfrak{m}_t \cong \mathbb{C}$.
\end{lemma}

\begin{proof}
Consider the evaluation morphism
\[
\operatorname{ev}_{\sigma}\colon
\mathcal E^\vee:=\mathcal H om_{\mathcal O_S}(\mathcal E,\mathcal O_S)
\to \mathcal O_S,
\qquad
\varphi\longmapsto \varphi(\sigma),
\]
and let $\mathcal I_{\sigma}:=\operatorname{Im}(\operatorname{ev}_{\sigma})$.
Then $\mathcal I_{\sigma}$ is a coherent ideal sheaf. Since $\mathcal E$ is
torsion-free, the natural morphism
$\iota:\mathcal E\to \mathcal E^{\vee\vee}$
is injective. Under the identification
$\mathcal E^{\vee\vee}:=\mathcal H om_{\mathcal O_S}(\mathcal E^\vee,\mathcal O_S)$,
the image of $\sigma$ under $\iota$ is precisely $\operatorname{ev}_{\sigma}$. Thus
$\sigma\neq 0$ implies $\operatorname{ev}_{\sigma}\neq 0$, and consequently
$\mathcal I_{\sigma}\neq 0$.

For each $t\in Z(\sigma)$, we have
$\sigma_t\in\mathfrak m_t\mathcal E_t$. Therefore, for each
$\varphi\in\operatorname{Hom}_{\mathcal O_{S,t}}(\mathcal E_t,\mathcal O_{S,t})$,
one has $\varphi(\sigma_t)\in\mathfrak m_t$. Thus
$(\mathcal I_{\sigma})_t\subset\mathfrak m_t$, so $t\in \operatorname{Supp}(\mathcal O_S/\mathcal I_{\sigma})$.
Hence 
\[Z(\sigma)\subset \operatorname{Supp}(\mathcal O_S/\mathcal I_{\sigma}).\]
Since $\mathcal I_{\sigma}$ is coherent, $\operatorname{Supp}(\mathcal O_S/\mathcal I_{\sigma})$ is an analytic
subset of $S$. Moreover, $\mathcal I_{\sigma}\neq 0$ and $S$ is connected, so
$V(\mathcal I_{\sigma})\neq S$. Therefore $V(\mathcal I_{\sigma})$ is a proper
analytic subset of $S$ containing $Z(\sigma)$.
\end{proof}

As a result, we obtain the following observation.

\begin{lemma}\label{lem-fiber-torsion}
Let $S\subset \mathbb C^m$ be a polydisc, let $F$ be a coherent sheaf
on $S$, and let $A\subset S$ be a subset that  is not contained in any proper analytic subset of $S$. 
Let  $s\in H^0(S,F)$
be a section such that
\[
s(t)=0
\qquad\text{in }F_t\otimes_{\mathcal O_{S,t}}k(t)
\]
for each $t\in A$.   Then  $s\in H^0(S,\operatorname{Tor}(F))$.
\end{lemma}
\begin{proof}
Set
\[
T:=\operatorname{Tor}(F),
\qquad
G:=F/T,
\]
and let \(q:F\to G\) be the natural quotient morphism. Denote by
\[
\theta:H^0(S,F)\to H^0(S,G)
\]
the induced morphism on global sections, and set
\[
\overline{s}:=\theta(s).
\]
For each \(t\in S\), let
\[
q_t:F_t\to G_t
\]
be the morphism on stalks induced by \(q\).

For each $t\in A$, the assumption
\[
s(t)=0
\qquad\text{in }F_t\otimes_{\mathcal O_{S,t}}k(t)
\]
is equivalent to $s_t\in\mathfrak m_tF_t$. Hence
\[
\overline{s}_t=q_t(s_t)\in q_t(\mathfrak m_tF_t)\subset \mathfrak m_tG_t.
\]
Therefore
\[
\overline{s}(t)=0
\qquad\text{in }G_t\otimes_{\mathcal O_{S,t}}k(t)
\]
for each $t\in A$.

Applying Lemma \ref{lem-zero-locus} to
the torsion-free  sheaf $G$ gives that  $\overline{s}=0.$
The  exact sequence
\[
0\to T\to F\to G\to 0
\]
gives a left exact sequence
\[
0\to H^0(S,T)\to H^0(S,F)
\to H^0(S,G).
\]
Thus $\ker\theta=H^0(S,T)$ as a subspace of $H^0(S,F)$. Since $\theta(s)=0$,
we get
\[
s\in H^0(S,T)=H^0(S,\operatorname{Tor}(F)).
\]
\end{proof}

\section{Lifting flat sections to real $d$-closed $(1,1)$-forms}

In this section, combining some observations on torsion sheaves with a method (via solving a $\bar\partial$-equation) for killing the $(0,2)$ and $(2,0)$ components of a topological lift of $\sigma$, we lift a flat section $\sigma$ whose restrictions to many fibers (not necessarily near the central fiber) are of type $(1,1)$ to a real $d$-closed $(1,1)$-form on the total space.

\begin{theorem}\label{thm-H11-lift}
Let $\pi:X\to S$ be a proper holomorphic submersion onto a polydisc
$S\subset \mathbb C^m$ centered at $0$ from a complex manifold $X$.
For each $t\in S$, set
\[
X_t:=\pi^{-1}(t)
\text{  and  }
k(t):=\mathcal O_{S,t}/\mathfrak m_t\cong \mathbb C.
\]
Let $\sigma\in H^0(S,R^2\pi_*\mathbb R)$
be a section, and for each $t\in S$, still denote by $\sigma_t\in H^2(X_t,\mathbb R)$
the image of the germ of $\sigma$ at $t$ under the canonical identification
\[
(R^2\pi_*\mathbb R)_t\cong H^2(X_t,\mathbb R).
\]
Set  $E:=R^2\pi_*\mathcal O_X.$
Let $V\subset S$ be a nonempty analytic Zariski open subset 
 such that
$E|_V$ is locally free, and, for each $t\in V$, the base-change morphism
\[
\operatorname{bc}_t:
E_t\otimes_{\mathcal O_{S,t}}k(t)
\to
H^2(X_t,\mathcal O_{X_t})
\]
is an isomorphism (such a $V$ exists).

Assume moreover that the following conditions hold:
\begin{enumerate}[\rm{(}1\rm{)}]
\item\label{cond-H11-det}
The restriction homomorphism
\[
H^0(S,E)\to H^0(V,E|_V)
\]
is injective.

\item\label{cond-H11-dense}
There exists a subset $W\subset V$, dense (equivalently, not contained in a proper analytic subset) in $S$ in the analytic Zariski topology,
such that, for each $t\in W$, the class $\sigma_t$ is represented by a smooth real
$d$-closed $(1,1)$-form on $X_t$.
\end{enumerate}
Then there exists a global smooth real $d$-closed $(1,1)$-form
$\widetilde\alpha$ on $X$ such that
\[
[\widetilde\alpha|_{X_t}]=\sigma_t
\qquad\text{for each }t\in S.
\]
\end{theorem}

\begin{proof}
\setcounter{proofstep}{0}

We first justify the parenthetical existence of  the  nonempty analytic Zariski open subset  $V$  in  Theorem \ref{thm-H11-lift}.
By Grauert's direct image theorem, $E$ is coherent.   By Grauert's upper semicontinuity theorem in the analytic Zariski topology
(e.g., \cite[\S 10.5.4]{GR84}), the function
$t\mapsto h^2(X_t,\mathcal O_{X_t})$ can jump only along a proper analytic
subset of $S$. Thus there exists a nonempty analytic Zariski open subset
$V\subset S$ such that $h^2(X_t,\mathcal O_{X_t})$ is constant for
$t\in V$.
By  Grauert's base-change theorem (e.g., \cite[Chapter I, (8.5) Theorem]{BHPV04}), 
for each $t\in V$, the base-change morphism
\[
\operatorname{bc}_t:
E_t\otimes_{\mathcal O_{S,t}}k(t)
\to
H^2(X_t,\mathcal O_{X_t})
\]
is an isomorphism and $E|_V$ is locally free.

The rest of the proof proceeds in the following four steps.  The strategy is to lift $\sigma$ topologically and then  kill the $(0,2)$ and $(2,0)$ components of  the topological lift of $\sigma$.

\begin{step}[Lifting  the flat section $\sigma$ to a $2$-class ${[\alpha_0]}$]\label{step-lift-1}

By Ehresmann's  theorem, $\pi$ is a $C^\infty$ locally trivial fibration.
Therefore
\[
V^q:=R^q\pi_*\mathbb R
\]
is a finite-rank local system on $S$ for each $q\geq 0$.
Since $S$ is a polydisc, it is contractible. Hence the local system $V^q$ is constant for each $q$. Thus
\[
V^q\cong \underline{V_q},
\qquad
V_q:=\left(R^q\pi_*\mathbb R\right)_0\cong H^q(X_0,\mathbb R).
\]

Note that the constant sheaf $\underline{V_q}$ admits the de Rham type fine resolution
\[
0 \to \underline{V_q}
\to \mathcal A_S^0\otimes_{\mathbb R}V_q
\xrightarrow{\,d\,}
\mathcal A_S^1\otimes_{\mathbb R}V_q
\xrightarrow{\,d\,}\cdots,
\]
where $\mathcal A_S^k$ denotes the sheaf of germs of  smooth real-valued $k$-forms on $S$.
Therefore
\[
H^p(S,R^q\pi_*\mathbb R)
=
H^p(S,\underline{V_q})
\cong
H^p\bigl(\Gamma(S,\mathcal A_S^\bullet\otimes_{\mathbb R}V_q)\bigr).
\]
Since $V_q$ is finite-dimensional over $\mathbb R$, we have
\[
H^p\bigl(\Gamma(S,\mathcal A_S^\bullet\otimes_{\mathbb R}V_q)\bigr)
\cong
H^p\bigl(\Gamma(S,\mathcal A_S^\bullet)\bigr)\otimes_{\mathbb R}V_q
\cong
H^p_{\mathrm{dR}}(S,\mathbb R)\otimes_{\mathbb R}V_q.
\]
Since $S$ is starshaped, the Poincar\'e lemma gives
\[
H^p_{\mathrm{dR}}(S,\mathbb R)=0
\qquad\text{for each }p\geq 1.
\]
Thus
\begin{equation}\label{formu-vani}
H^p(S,R^q\pi_*\mathbb R)=0
\qquad\text{for each }p\geq 1\text{ and each }q\geq 0.
\end{equation}
Consequently, applying Lemma \ref{lem-leray-edge}  implies that the  edge homomorphism
\[
\lambda_2: H^2( X,\mathbb R)
\to
H^0(S,R^2\pi_*\mathbb R)
\]
is an isomorphism.

Let  $c:=\lambda_2^{-1}(\sigma).$  By deRham isomorphism theorem and Lemma \ref{lem-H2-lift}, we may 
choose a smooth real $d$-closed $2$-form $\alpha_0$ on $X$ representing $c$ such that 
\[
[\alpha_0|_{X_t}]=\sigma_t.
\]

\end{step}

\begin{step}[The  obstruction to killing the $(0,2)$-component of $\alpha_0$]

Since $d\alpha_0=0$, the $(0,3)$-component of $d\alpha_0$ gives $\bar\partial\alpha_0^{0,2}=0.$
Hence $\alpha_0^{0,2}$ defines a Dolbeault class
\[
[\alpha_0^{0,2}]\in H^2(X,\mathcal O_X).
\]
Note that the class $[\alpha^{0,2}_0]$ is independent of the chosen real closed representative $\alpha_0$ of $c$. Indeed, replacing $\alpha_0$ by $\alpha_0+d\beta$ changes $\alpha^{0,2}_0$ by $\overline{\partial}\beta^{0,1}$.

For each $q\geq 0$, Grauert's direct image theorem implies that
$R^q\pi_*\mathcal O_X$ is coherent. Since $S$ is Stein, Cartan's theorem B gives
\[
H^a(S,R^q\pi_*\mathcal O_X)=0
\qquad\text{for each }a>0\text{ and each }q\geq 0.
\]
Therefore Lemma \ref{lem-leray-edge} gives an isomorphism  
\[
\rho:H^2(X,\mathcal O_X)\xrightarrow{\cong}H^0(S,E).
\]
Set
\[
\nu:=\rho([\alpha_0^{0,2}])\in H^0(S,E).
\]

Thus the Dolbeault cohomology class $[\alpha^{0,2}_0]$ (or $\nu$) is exactly the obstruction to modifying $\alpha_0$ by a real exact form so that the resulting representative has no $(0,2)$-component.

\end{step}

\begin{step}[Vanishing of the obstruction section $\nu$]\label{H11lift-step3}

In this step, we  prove that $\nu=0.$
Fix any $t\in W$. By condition \eqref{cond-H11-dense}, choose a smooth real
$d$-closed $(1,1)$-form $\eta_t$ on $X_t$ such that $[\eta_t]=\sigma_t.$
Since
\[
[\alpha_0|_{X_t}]=\sigma_t,
\]
there exists a smooth real $1$-form $\gamma_t$ on $X_t$ such that
\[
\alpha_0|_{X_t}-\eta_t=d\gamma_t.
\]
Taking the $(0,2)$-component gives
\[
\alpha_0^{0,2}|_{X_t}
=
\bar\partial\gamma_t^{0,1}.
\]
Thus
\[
[\alpha_0^{0,2}|_{X_t}]=0
\qquad\text{in }H^2(X_t,\mathcal O_{X_t}).
\]

By Lemma \ref{lem-bc-edge},  this restriction class is equal to 
\[
\operatorname{bc}_t(\operatorname{ev}_t(\nu)).
\]
Since $t\in W\subset V$, the morphism $\operatorname{bc}_t$ is an isomorphism. Hence the value  of $\nu$ at $t\in W$ vanishes:
\[
\operatorname{ev}_t(\nu)=0
\qquad\text{in }E_t\otimes_{\mathcal O_{S,t}}k(t)
\]
for each $t\in W$.

Since $W$ is dense in $S$ in the analytic Zariski topology,
Lemma \ref{lem-fiber-torsion}  implies that
\[
\nu\in H^0(S,\operatorname{Tor}(E)).
\]
Since $E|_V$ is locally free, one has
\[
\operatorname{Tor}(E)|_V=0.
\]
Thus
\[
\nu|_V=0\in H^0(V,E|_V).
\]
By condition  \eqref{cond-H11-det}, the restriction homomorphism
\[
H^0(S,E)\to H^0(V,E|_V)
\]
is injective. Therefore
\[
\nu=0\in H^0(S,E).
\]
Since $\rho$ is an isomorphism, it follows that
\[
[\alpha_0^{0,2}]=0
\qquad\text{in }H^2(X,\mathcal O_X).
\]
\end{step}

\begin{step}[Construction of the desired $(1,1)$-form]
By the Dolbeault isomorphism theorem, there exists a smooth $(0,1)$-form $\theta$ on $X$ such that
\[
\bar\partial\theta=\alpha_0^{0,2}.
\]
Define
\[
\widetilde\alpha:=\alpha_0-d(\theta+\overline{\theta}).
\]
Then $\widetilde\alpha$ is a smooth real $d$-closed $2$-form on $X$, and
$\widetilde\alpha-\alpha_0$ is $d$-exact. Hence
\[
[\widetilde\alpha|_{X_t}]
=
[\alpha_0|_{X_t}]
=
\sigma_t
\qquad\text{for each }t\in S.
\]

The $(0,2)$-component of $d\theta$ is $\bar\partial\theta$, while
$d\overline{\theta}$ has no $(0,2)$-component. Therefore
\[
\widetilde\alpha^{0,2}
=
\alpha_0^{0,2}-\bar\partial\theta
=
0.
\]
Since $\widetilde\alpha$ is real, one also has
\[
\widetilde\alpha^{2,0}
=
\overline{\widetilde\alpha^{0,2}}
=
0.
\]
Thus the only possible nonzero component of $\widetilde\alpha$ is its $(1,1)$-component.
Therefore $\widetilde\alpha$ is a smooth real $d$-closed $(1,1)$-form.
This completes the proof.
\end{step}
\end{proof}

When $R^2\pi_*\mathcal O_X$ is torsion free, we obtain the following equivalent characterization of the liftability of a flat section to a global smooth real $d$-closed $(1,1)$-form.

\begin{corollary}\label{cor-H11-tf}
Let $\pi:X\to S$ be a proper holomorphic submersion onto a polydisc
$S\subset \mathbb C^m$ from a complex manifold $X$, with fibers denoted by $X_t:=\pi^{-1}(t)$.
Let $\sigma\in H^0(S,R^2\pi_*\mathbb R)$
be a section.   Denote by $\sigma_t\in H^2(X_t,\mathbb R)$
the image of the germ  $\sigma_t$  under the canonical identification
$(R^2\pi_*\mathbb R)_t\cong H^2(X_t,\mathbb R).$ 
Assume that $E:=R^2\pi_*\mathcal O_X$  is torsion-free.  

 Then the following conditions are equivalent:
\begin{enumerate}[\rm{(}1\rm{)}]
\item\label{cond-tf-dense}
There exists a subset $U\subset S$, dense in $S$ in the analytic Zariski topology,
such that, for each $t\in U$, the class
$\sigma_t\in H^2(X_t,\mathbb{R})$ is represented by a smooth real
$d$-closed $(1,1)$-form on $X_t$.

\item\label{cond-tf-global}
There exists a global smooth real $d$-closed $(1,1)$-form
$\widetilde{\alpha}$ on $X$ such that
\[
[\widetilde{\alpha}|_{X_t}]=\sigma_t
\qquad\text{for each }t\in S.
\]
\end{enumerate}
\end{corollary}

\begin{remark}
    Note that if a fiber $X_b$ satisfies $h^{0,1}$-invariance, i.e.,  the function
$t\mapsto h^1(X_t,\mathcal{O}_{X_t})$ is locally constant in a
neighbourhood of $b$ (e.g., when the Fr\"olicher spectral sequence of $X_b$ degenerates at $E_1$), then  $R^2\pi_*\mathcal{O}_X$ is torsion-free near $b$ (\cite[p. 211, \S 10.5.5]{GR84}).
\end{remark}

\begin{proof}
Condition \eqref{cond-tf-global} implying condition
\eqref{cond-tf-dense}  is clear.
Conversely, assume condition \eqref{cond-tf-dense}.

As in the first paragraph of the proof of Theorem \ref{thm-H11-lift}, the dense analytic Zariski open subset $V$ appearing in Theorem \ref{thm-H11-lift} exists: $E|_V$ is locally free and,
for each $t\in V$, the base-change morphism
\[
\operatorname{bc}_t:
E_t\otimes_{\mathcal O_{S,t}}k(t)
\to
H^2(X_t,\mathcal O_{X_t})
\]
is an isomorphism. Moreover, the complement 
\[B:= S\setminus V\]
is a proper analytic subset of $S$.

We will apply Theorem \ref{thm-H11-lift} to give the desired 
  $\widetilde{\alpha}$ in condition \eqref{cond-tf-global}.
We now verify condition \eqref{cond-H11-det} of Theorem \ref{thm-H11-lift}.
Let $\tau\in H^0(S,E)$
be a section such that $\tau|_V=0.$
In particular, for each $t\in V$, the germ $\tau_t$ is zero in $E_t$ and thus the value
\[
\operatorname{ev}_t(\tau)=0
\quad\text{in }E_t\otimes_{\mathcal{O}_{S,t}}k(t)
\qquad\text{for each }t\in V.
\]
Applying Lemma \ref{lem-fiber-torsion} gives that
\[
\tau\in H^0(S,\operatorname{Tor}(E)).
\]
Since $E$ is torsion-free, one has $\tau=0.$
Therefore the restriction homomorphism
\[
H^0(S,E)\to H^0(V,E|_V)
\]
is injective.

Set
\[
W:=U\cap V.
\]
We claim that $W$ is dense in $S$ in the analytic Zariski topology.   In fact, if $W$ were not dense in $S$, then
there would exist a proper analytic subset $A\subsetneq S$ such that $W\subset A.$
Thus $U\subset A\cup B.$
However $A\cup B$ is  a proper
analytic subset of $S$. This contradicts the analytic Zariski density of $U$.
Therefore $W$ is dense in $S$ in the analytic Zariski topology.

For each $t\in W$, one has $t\in U$, and hence $\sigma_t$ is represented by a
smooth real $d$-closed $(1,1)$-form on $X_t$. Thus condition
\eqref{cond-H11-dense} of Theorem \ref{thm-H11-lift} is satisfied.
Consequently, Theorem \ref{thm-H11-lift} gives a global smooth real $d$-closed $(1,1)$-form
$\widetilde{\alpha}$ on $X$ such that
\[
[\widetilde{\alpha}|_{X_t}]=\sigma_t
\qquad\text{for each }t\in S.
\]
Thus condition \eqref{cond-tf-global} holds. This completes the proof.
\end{proof}

\section{K\"ahlerness  of the total space near a K\"ahler fiber}

\subsection{Existence of a K\"ahler neighborhood of a compact K\"ahler submanifold}

We first establish the following  auxiliary lemma, which provides a smooth distance-square type function whose Levi form is semipositive along the submanifold and strictly positive in the normal directions.

\begin{lemma}\label{lem-rho}
Let $X$ be a complex manifold and let $Z\subset X$ be a (closed) complex
submanifold. Then there exists a smooth real-valued function
$\rho\in C^{\infty}(X,\mathbb R)$ such that:
\begin{enumerate}[\rm{(}1\rm{)}]
\item $\rho\geq 0$ on $X$;
\item $\rho^{-1}(0)=Z$;
\item $d\rho_{z}=0 \in T_{X, z}^*$  for any $z\in Z$;
\item for each $x\in Z$, the Levi form
$\bigl(\sqrt{-1}\partial\bar\partial\rho\bigr)_{x}$ (acting on $T_{X,x}^{1,0}$) is semipositive;
\item for each $x\in Z$, the  form induced by
$\bigl(\sqrt{-1}\partial\bar\partial\rho\bigr)_{x}$ on
$N_{Z/X,x}=T_{X,x}^{1,0}/T_{Z,x}^{1,0}$ is strictly positive definite;
\item 
$\sqrt{-1}\partial\bar\partial\rho\bigr|_{Z}=0$.
\end{enumerate}
\end{lemma}

\begin{proof}
Choose a Hermitian metric $\omega$ on $X$, and let $g$ be the associated
Riemannian metric. Then there exists an open tubular neighborhood
$U$ of $Z$ such that the squared distance function
\[
h:=\operatorname{dist}_{g}(\,\cdot\,,Z)^{2}
\]
is smooth on $U$ (\cite{M92}). Moreover,
\[
h\geq 0,\qquad h^{-1}(0)=Z,\qquad dh_{z}=0  \text{ for any } z\in Z.
\]

For each $p\in Z$, apply   \cite[Proposition 2.1]{NWZ24}  to the complex submanifold
$Z\subset X$ with the Hermitian metric $\omega$. Let
$k=\dim_{\mathbb C,p}Z$ and $r=\operatorname{codim}_{\mathbb C,p}(Z,X)$.
Then there are holomorphic coordinates
\[
(z^{1},\ldots,z^{k},w^{1},\ldots,w^{r})
\]
centered at $p$ such that $Z$ is locally given by
$w^{1}=\cdots=w^{r}=0$, and the Levi matrix of $h$ at $p$ has the  form
\[
\left(
\frac{\partial^{2}h}
{\partial \xi^{\alpha}\partial\overline{\xi^{\beta}}}(p)
\right)
=
\begin{pmatrix}
0 & 0\\
0 & I_{r}
\end{pmatrix},
\qquad \xi=(z,w).
\]
Consequently, $\bigl(\sqrt{-1}\partial\bar\partial h\bigr)_{p}$ is
semipositive on $T_{X,p}^{1,0}$. Moreover, its kernel  is $T_{Z,p}^{1,0}$,
and the
Hermitian form induced by $\bigl(\sqrt{-1}\partial\bar\partial h\bigr)_{p}$
on
\[
N_{Z/X,p}=T_{X,p}^{1,0}/T_{Z,p}^{1,0}
\]
is positive definite.

It remains to extend $h$ to a global smooth function without changing it
near $Z$. Since $Z$ is closed and $U$ is an open neighborhood of $Z$, choose
a smooth function $\eta\in C^{\infty}(X,[0,1])$ such that $\eta\equiv 1$ on
an open neighborhood of $Z$ and $\operatorname{supp}(\eta)\subset U$. Since
$\operatorname{supp}(\eta)\subset U$, the function $\eta h$ on $U$ extends
by zero to a smooth function on $X$. We still denote this extension by
$\eta h$, and define
\[
\rho:=\eta h+(1-\eta).
\]

Then $\rho\geq 0$ on $X$. If $x\notin U$, then $\eta(x)=0$, and so
$\rho(x)=1$. If $x\in U$ and $\eta(x)<1$, then
\[
\rho(x)=\eta(x)h(x)+1-\eta(x)>0.
\]
If $x\in U$ and $\eta(x)=1$, then $\rho(x)=h(x)$, which vanishes precisely
when $x\in Z$. Thus $\rho^{-1}(0)=Z$.

Finally, since $\eta\equiv 1$ on a neighborhood of $Z$, one has $\rho=h$ on
a neighborhood of $Z$. Therefore
\[
d \rho_x=d h_x=0 \in T_{X, x}^* \quad \text { for each } x \in Z, 
\]
and for each $p\in Z$,
\[
\bigl(\sqrt{-1}\partial\bar\partial\rho\bigr)_{p}
=
\bigl(\sqrt{-1}\partial\bar\partial h\bigr)_{p}.
\]
The semipositivity on $T_{X,p}$ and the positive definiteness on
$N_{Z/X,p}$ therefore follow  from the corresponding properties of $h$.

Moreover, since $\rho|_{Z}=0$,  we have
\[
\sqrt{-1}\partial\bar\partial\rho\bigr|_Z
=
\sqrt{-1}\partial\bar\partial(i^{*}\rho)
=
0,
\]
where $i$ is  the inclusion $i:Z\hookrightarrow X$.
This proves the lemma.
\end{proof}

We now prove the following  criterion for the existence of a K\"ahler neighborhood of a compact K\"ahler  submanifold.

\begin{theorem}\label{thm-pair-kahler}
Let $X$ be a complex manifold, and let $Z\subset X$ be a compact K\"ahler
submanifold. Let $U$ be an open neighborhood of $Z$ in $X$, and let
$\alpha$ be a smooth real $d$-closed $(1,1)$-form on $U$. Assume that there exists a K\"ahler form $\omega_Z$ on $Z$ such that \[
[\omega_Z]=[i^*\alpha]\quad\text{in }H_{\rm dR}^2(Z,\mathbb R),
\] where $i:Z\hookrightarrow U$ is the natural inclusion.
Then there exist an open neighborhood $W\subset U$ of $Z$ and a smooth
real-valued function $\Psi\in C^\infty(W,\mathbb R)$
such that
\[
\widehat{\alpha}:=\alpha|_W+\sqrt{-1}\partial\bar\partial\Psi
\]
is a K\"ahler form on $W$ and
\[
i_W^*\widehat{\alpha}=\omega_Z,
\]
where $i_W:Z\hookrightarrow W$ is the inclusion. 
\end{theorem}

\begin{proof}
 Since
\[
[\omega_Z]=[i^*\alpha]\quad\text{in }H_{\rm dR}^2(Z,\mathbb R),
\]
the form $\omega_Z-i^*\alpha$ is  real $d$-exact on $Z$.
By  $\partial\bar\partial$-lemma on the compact K\"ahler manifold $Z$, there exists a smooth real-valued
function $\varphi\in C^\infty(Z,\mathbb R)$
such that
\[
\omega_Z-i^*\alpha
=
\sqrt{-1}\partial_Z\bar\partial_Z\varphi.
\]
Extend $\varphi$ to a smooth
real-valued function $\Phi\in C^\infty(U,\mathbb R).$
Set
\[
\alpha_1:=\alpha+\sqrt{-1}\partial\bar\partial\Phi.
\]
Then $\alpha_1$ is a smooth real $d$-closed $(1,1)$-form on $U$ 
 such that 
 \[i^*\alpha_1=\omega_Z.\]

By Lemma \ref{lem-rho}, applied to the pair $(U,Z)$, there exists a smooth
real-valued function $\rho\in C^\infty(U,\mathbb R)$ such that, for
\[
\lambda:=\sqrt{-1}\partial\bar\partial\rho,
\]
the  form $\lambda_z$ is semipositive on $T_{U,z}^{1,0}$ for
each $z\in Z$, the induced  form on
\[
N_{Z/U,z}:=\left(T_U^{1,0}|_Z/T_Z^{1,0}\right)_z
\]
is strictly positive definite for each $z\in Z$, and $i^*\lambda=0.$

Utilizing the fact that $i^*\alpha_1=\omega_Z$ is a K\"ahler form on $Z$ and the properties of $\rho$, we obtain the following claim by elementary linear algebra theory.

\begin{claim}
   There exists $C>0$ such that
\[
\alpha_1+C\lambda
\]
is strictly  positive definite on $T_{U,z}^{1,0}$ for each $z\in Z$.
\end{claim}

\begin{proof}[Proof of the claim]
 For the reader's convenience, we write the  details of this claim.
 
Fix any smooth Hermitian metric $h$ on
$U$.  Then $N_{Z/U}:=T_U^{1,0}|_Z/T_Z^{1,0}$, endowed with   the quotient metric, 
can be viewed as an $h$-orthogonal complement of $T_Z^{1,0}$ in
$T_U^{1,0}|_Z$, as $\mathscr{C}^{\infty}$ complex vector bundles over $Z$.

Since $i^*\lambda=0$, one has
$\lambda_z(\tau,\overline{\tau})=0$
for each $z\in Z$ and each $\tau\in T_{Z,z}^{1,0}$. Since $\lambda_z$ is
semipositive, the Cauchy--Schwarz inequality for semipositive Hermitian
forms gives
\[
\left|\lambda_z(\tau,\overline{\xi})\right|^2
\leq
\lambda_z(\tau,\overline{\tau})\lambda_z(\xi,\overline{\xi})
=0
\]
for each $\xi\in T_{U,z}^{1,0}$. Thus
\[
\lambda_z(\tau,\overline{\xi})=0
\]
for each $z\in Z$, each $\tau\in T_{Z,z}^{1,0}$, and each
$\xi\in T_{U,z}^{1,0}$.

Therefore, with respect to the smooth splitting
\[
T_{U,z}^{1,0}=T_{Z,z}^{1,0}\oplus N_{Z/U,z},
\]
the Hermitian matrix of $\alpha_1+C\lambda$ has the  form
\[
\begin{pmatrix}
A_z & B_z\\
B_z^* & D_z+C Q_z
\end{pmatrix}.
\]
Here all the matrix entries depend smoothly on $z$,  $A_z$ is the strictly positive definite Hermitian matrix of $\omega_{Z,z}$  on
$T_{Z,z}^{1,0}$ (since $i^*\alpha_1=\omega_Z$),  and $Q_z$ is the Hermitian matrix of
$\lambda_z|_{N_{Z/U,z}}$.

Since $A_z>0$, $Q_z>0$ and $Z$ is compact, we obtain  that  there exists a constant $C>0$ such that 
$(\alpha_1+C\lambda)_z$
is strictly positive definite on $T_{U,z}^{1,0}$ for each $z\in Z$. This proves the claim.
\end{proof}

Set \[\Psi:=(\Phi+C \rho)\]  and 
\[
\alpha_C:=\alpha_1+C\lambda.
\]
Since strict positive-definiteness is an open condition, after shrinking $U$ to an open
neighborhood $W$ of $Z$, the form
\[
\widehat{\alpha}:=\alpha_C|_W
\]
is strictly positive definite on $W$.  Clearly,  $\widehat{\alpha}$ is the desired K\"ahler form on $W$.
This completes the proof.
\end{proof}

\subsection{Equivalent characterization of the K\"ahlerness of the total space near a K\"ahler fiber}

As a  consequence of Corollary \ref{cor-H11-tf} and Theorem
\ref{thm-pair-kahler}, we obtain an  equivalent characterization of the K\"ahlerness of the total space near a K\"ahler fiber.

We first establish the following auxiliary observation, which is utilized to deal with the case where
$U$ is  far away from  $0$ in 
condition \eqref{cond-lk-section} of Corollary \ref{cor-local-kahler}.

\begin{lemma}\label{lem-coeff-fiber}
Let $\pi:X\to S$ be a proper holomorphic submersion from a complex
manifold onto a polydisc $S\subset \mathbb C^m$ centered at $0$. For each $t\in S$, set
$X_t:=\pi^{-1}(t), 
k(t):=\mathcal O_{S,t}/\mathfrak m_t, L:=R^2\pi_*\mathbb R_X$  and $
E:=R^2\pi_*\mathcal O_X.$

Let $\Phi:L\to E$
be the morphism  induced by the natural 
morphism $\theta: \mathbb R_X\to\mathcal O_X$. Let $\sigma\in H^0(S,L)$ and set
\[
\eta:=H^0(S,\Phi)(\sigma)\in H^0(S,E),
\]
where $H^0(S,\Phi)$ denotes the induced morphism on global sections associated with the sheaf morphism $\Phi$.
For each $t\in S$, let
\[
\sigma_t\in H^2(X_t,\mathbb R)
\]
be the class obtained from the value of $\sigma$ at $t$ under the canonical
identification
\[
(R^2\pi_*\mathbb R_X)_t\cong H^2(X_t,\mathbb R).
\]
Let
\[
q_t:H^2(X_t,\mathbb R)
\to
H^2(X_t,\mathcal O_{X_t})
\]
be the  map induced by
$\mathbb R_{X_t}\to\mathcal O_{X_t}$. Let
\[
\operatorname{ev}_t:H^0(S,E)
\to
E_t\otimes_{\mathcal O_{S,t}}k(t)
\]
be evaluation at $t$, and let
\[
\operatorname{bc}_t:
E_t\otimes_{\mathcal O_{S,t}}k(t)
\to
H^2(X_t,\mathcal O_{X_t})
\]
be the base-change morphism. Then
\[
\operatorname{bc}_t\bigl(\operatorname{ev}_t(\eta)\bigr)
=
q_t(\sigma_t)
\]
for each $t\in S$. In particular, if $\eta=0$, then
\[
q_t(\sigma_t)=0
\quad\text{in}\quad
H^2(X_t,\mathcal O_{X_t})
\]
for each $t\in S$.
\end{lemma}

\begin{proof}
As in Step \ref{step-lift-1} of the proof of Theorem \ref{thm-H11-lift},  the
 edge homomorphism
\[
\rho_{\mathbb R}:H^2(X,\mathbb R)
\xrightarrow{\ \cong\ }
H^0(S,L)
\]
of  the Leray  spectral sequence is an isomorphism. Therefore there exists a unique class
\[
c\in H^2(X,\mathbb R)
\]
such that $\rho_{\mathbb R}(c)=\sigma.$
Moreover, Lemma \ref{lem-H2-lift}  gives
\[
c|_{X_t}=\sigma_t
\quad\text{in}\quad
H^2(X_t,\mathbb R)
\]
for each $t\in S$.

Let
\[
q_X:=H^2(X,\theta):
H^2(X,\mathbb R_X)\to H^2(X,\mathcal O_X)
\]
be induced by $\theta$.
Apply the functoriality (e.g.,  \cite[Chapter I, \S 1, p. 13]{BS76} or  \cite[Tag
01F6]{Stacks}) of the Leray spectral sequence  to $\theta$, the following diagram commutes:
\[
\begin{array}{ccc}
H^2(X,\mathbb R_X)
&
\xrightarrow{\ \rho_{\mathbb R}\ }
&
H^0(S,R^2\pi_*\mathbb R_X)
\\[1.4em]
\Big\downarrow\vcenter{\rlap{$\scriptstyle q_X$}}
&
&
\Big\downarrow\vcenter{\rlap{$\scriptstyle H^0(S,\Phi)$}}
\\[1.4em]
H^2(X,\mathcal O_X)
&
\xrightarrow{\ \rho_{\mathcal O}\ }
&
H^0(S,R^2\pi_*\mathcal O_X)
\end{array}
\]
where $\rho_{\mathcal O}$ denotes the  edge homomorphism of the Leray spectral sequence for $\mathcal O_X$.
Consequently, one has
\[
\rho_{\mathcal O}\bigl(q_X(c)\bigr)
=
H^0(S,\Phi)(\sigma)
=
\eta.
\]

Applying Lemma \ref{lem-bc-edge} gives that
\[
\operatorname{bc}_t
\bigl(
\operatorname{ev}_t(\rho_{\mathcal O}(q_X(c)))
\bigr)
=
q_X(c)|_{X_t}.
\]
Consequently, 
\[
\operatorname{bc}_t\bigl(\operatorname{ev}_t(\eta)\bigr)
=
\bigl(q_X(c)\bigr)|_{X_t}=q_t(c|_{X_t})=q_t(\sigma_t).
\]
This completes the proof.
\end{proof}

Now we establish the K\"ahlerness of the total space near a K\"ahler fiber.

\begin{corollary}\label{cor-local-kahler}
Let $\pi:X\to S$ be a proper holomorphic submersion from a complex
manifold $X$ onto a polydisc $S\subset\mathbb C^m$ centered at
$0$. Set $X_t:=\pi^{-1}(t)$ for each $t\in S$. Assume that
$X_0$ is K\"ahler.
Then the following conditions are equivalent:
\begin{enumerate}[\rm{(}1\rm{)}]
\item\label{cond-lk-total}
There exists a polydisc $S_1\Subset S$ centered at $0$ such that
\[
X_{S_1}:=\pi^{-1}(S_1)
\]
is K\"ahler.

\item\label{cond-lk-section}
There exist a polydisc $S_2\Subset S$ centered at $0$ and a section
\[
\sigma\in H^0(S_2,(R^2\pi_*\mathbb R)|_{S_2})
\]
such that, under the canonical identification
\[
(R^2\pi_*\mathbb R)_t\cong H^2(X_t,\mathbb R),
\]
one has
$\sigma_0=[\omega_0]$
for some K\"ahler form $\omega_0$ on $X_0$, and there exists a subset
$U\subset S_2$ (note that $U$ may even be disjoint from a small open neighborhood of $0$), dense in the analytic Zariski topology of $S_2$, such that, for
each $t\in U$, the class $\sigma_t\in H^2(X_t,\mathbb R)$
is represented by a smooth real $d$-closed $(1,1)$-form on $X_t$.
\end{enumerate}
\end{corollary}

\begin{remark}\label{rem-lk-moreover}
In the implication from condition \eqref{cond-lk-section} to condition
\eqref{cond-lk-total}, the polydisc $S_1$ may be chosen inside $S_2$ in
such a way that, for each $t\in S_1$, the class $\sigma_t$ is a
K\"ahler class on $X_t$. Indeed, the K\"ahler form constructed on
$X_{S_1}$ has fiberwise cohomology class $\sigma_t$.
To see this, Corollary \ref{cor-H11-tf} gives a smooth real $d$-closed $(1,1)$-form whose restriction to $X_t$ represents $\sigma_t$, and Theorem \ref{thm-pair-kahler} modifies this form by a $\sqrt{-1}\partial\overline{\partial}$-term (after shrinking the base), hence without changing the  cohomology class of its restriction to each fiber.
\end{remark}

\begin{remark}\label{rem-tk-zariski}
In the implication from condition \eqref{cond-lk-section} to condition
\eqref{cond-lk-total}, the subset $U$ in condition \eqref{cond-lk-section} of Corollary \ref{cor-local-kahler}  is used only to invoke
Corollary \ref{cor-H11-tf}.   Under the condition \eqref{cond-lk-section}, it  may happen that 
$U\cap S_1=\emptyset$. This causes no difficulty, because the role of
$U$ is to produce the global smooth real $d$-closed $(1,1)$-form
$\widetilde\alpha$; the  K\"ahlerness statement of $X$ over $S_1$ is then obtained by a local argument around the compact
submanifold $X_0$.
\end{remark}

\begin{proof}[Proof of Corollary \ref{cor-local-kahler}]
Condition \eqref{cond-lk-total} clearly implies condition \eqref{cond-lk-section}.

Conversely, assume condition \eqref{cond-lk-section}. Denote by
\[
\pi_2:X_{S_2}:=\pi^{-1}(S_2)\to S_2
\]
the restricted family, and set
\[
E_2:=R^2(\pi_2)_*\mathcal O_{X_{S_2}}.
\]
Under the canonical restriction isomorphism
\[
(R^2\pi_*\mathbb R)|_{S_2}
\cong
R^2(\pi_2)_*\mathbb R,
\]
let
\[
\sigma_2:=\sigma|_{S_2}\in H^0(S_2,R^2(\pi_2)_*\mathbb R).
\]
We denote by
\[
\eta\in H^0(S_2,E_2)
\]
the image of $\sigma_2$ under the natural morphism induced by
\[
R^2(\pi_2)_*\mathbb R
\to
R^2(\pi_2)_*\mathcal O_{X_{S_2}}=E_2 .
\]

As in the first paragraph of the proof of Theorem \ref{thm-H11-lift}, there exists a dense analytic Zariski open subset  \[V\subset S_2\] such that $E_2|_V$ is locally free
and, for each $t\in V$, the base-change morphism
\[
bc_t: (E_2)_t\otimes_{\mathcal O_{S_2,t}}k(t)
\to
H^2(X_t,\mathcal O_{X_t})
\]
is an isomorphism.  Moreover, the complement  $S_2\setminus V$
is a proper analytic subset of $S_2$.
Consequently, since $U$ is dense in the analytic Zariski topology on $S_2$,  
the subset $U\cap V$ is not contained in any proper analytic subset of
$S_2$.
Recall the condition that for each $t\in U\cap V$, the class $\sigma_t$ is represented by a
smooth real $d$-closed $(1,1)$-form on $X_t$. 
Then  its image under  $q_t:H^2(X_t,\mathbb{R})\to H^2(X_t,\mathcal{O}_{X_t})$
is zero, where $q_t$ is induced by $\mathbb R_{X_t}\to\mathcal O_{X_t}$. By Lemma \ref{lem-coeff-fiber}, we have
\[
        \operatorname{bc}_t\bigl(\operatorname{ev}_t(\eta)\bigr)
        =
        q_t(\sigma_t)
        =
        0
\]
for each $t\in U\cap V$.
Since  the base-change morphism $\operatorname{bc}_t$ is an
isomorphism for $t\in V$,  $ \operatorname{ev}_t(\eta)=0.$
Since $U\cap V$ is dense in the analytic Zariski topology of $S_2$,
Lemma \ref{lem-fiber-torsion} gives
$
        \eta\in H^0\bigl(S_2,\operatorname{Tor}(E_2)\bigr).$
As $E_2|_V$ is locally free, one has
$ \operatorname{Tor}(E_2)|_V=0,$
and therefore
  \[
\eta|_V=0\in H^0(V,E_2|_V).
\]

Since $X_0$ is K\"ahler, the Kodaira--Spencer stability theorem and the deformation invariance of the Hodge numbers for fiberwise K\"ahler families imply
that, after replacing $S_2$ by a smaller polydisc \[S_3\Subset S_2\]
centered at $0$, the fiber $X_t$ is K\"ahler for each $t\in S_3$  and  the Hodge number
$h^{0,2}(X_t)$ is constant on $S_3$. By the Grauert's base change theorem (e.g., \cite[Chapter I, (8.5) Theorem]{BHPV04}),
\[
E_3:=R^2(\pi_3)_*\mathcal O_{X_{S_3}}
\cong
E_2|_{S_3}
\]
is locally free.

Note that $(\eta|_{S_3})|_{V\cap S_3}=0$ and 
$V\cap S_3$ is a nonempty open subset of $S_3$.  Since $E_3$
is locally free and $S_3$ is connected, we have that
\[
\eta|_{S_3}=0.
\]
 Applying Lemma \ref{lem-coeff-fiber} to the restricted
family $\pi_3:X_{S_3}\to S_3$
gives that
\[
q_t(\sigma_t)=0
\quad\text{in}\quad
H^2(X_t,\mathcal O_{X_t})
\]
for each $t\in S_3$.
Since $X_t$ is K\"ahler for each $t\in S_3$, the Hodge decomposition and $q_t(\sigma_t)=0$  give that the $(0,2)$-part of
$\sigma_t$ is zero. Since $\sigma_t$ is real, its $(2,0)$-part  is also zero. Hence the class $\sigma_t$ is represented by a smooth
real $d$-closed $(1,1)$-form on $X_t$ for each $t\in S_3$.

Now Corollary \ref{cor-H11-tf} applies to the restricted family
\[
\pi_3:X_{S_3}\to S_3
\]
and to the restricted section $\sigma|_{S_3}$. Then there exists a smooth real
$d$-closed $(1,1)$-form $\widetilde\alpha$ on $X_{S_3}$ such that
\[
[\widetilde\alpha|_{X_t}]=\sigma_t
\qquad\text{for each }t\in S_3.
\]
In particular,
\[
[\widetilde\alpha|_{X_0}]
=
\sigma_0
=
[\omega_0],
\]
where $\omega_0$ is the K\"ahler form on $X_0$ appearing in condition
\eqref{cond-lk-section}.

We now apply Theorem \ref{thm-pair-kahler} to the complex manifold
$X_{S_3}$ and  the compact complex submanifold $X_0\subset X_{S_3}$, to obtain that
 there exist an open neighborhood $N\subset X_{S_3}$ of $X_0$ and a
smooth real-valued function $\Psi\in C^\infty(N,\mathbb R)$ such that
\[
\widehat\omega
:=
\widetilde\alpha|_N+\sqrt{-1}\partial\bar\partial\Psi
\]
is a K\"ahler form on $N$.

 Since $\pi_3$ is proper, 
 the image $\pi_3(X_{S_3}\setminus N)$
is closed in $S_3$ such that 
\[
0\notin \pi_3(X_{S_3}\setminus N).
\]
Therefore we may choose a polydisc $S_1\Subset S_3$ centered at $0$ such
that
\[
S_1\cap \pi_3(X_{S_3}\setminus N)=\emptyset.
\]
Then
\[
X_{S_1}=\pi^{-1}(S_1)\subset N.
\]
Consequently, the restriction $\widehat\omega|_{X_{S_1}}$ is a
K\"ahler form on $X_{S_1}$. This proves condition
\eqref{cond-lk-total} and thus completes the proof.
\end{proof}

\section{Examples: sharpness of the type-$(1,1)$ restriction condition}

In this section, we will show that the type-$(1,1)$ restriction condition in Theorem \ref{intro-cor-local-kahler} cannot, in general, be imposed only on a proper analytic subset of the base.  We first construct a proper holomorphic submersion between
complex manifolds whose fibers are all  $K3$ surfaces,  by constructing a local period curve.  The resulting morphism carries a flat real cohomology class whose value on the central fiber is a K\"ahler class, while its values on all nearby fibers fail to be of type $(1,1)$.  We then show that, although the central fiber is K\"ahler, the total space is not K\"ahler after any shrinking of the base. Furthermore, by taking products with a polydisc, this one-parameter example yields examples over higher-dimensional bases in which the type-$(1,1)$ condition holds along a divisor of the base, while the total space still fails to be K\"ahler near the central K\"ahler fibers.

\subsection{Construction of a family of $K3$ surfaces}

The construction is based on the basic period theory of $K3$ surfaces (\cite[Chapter 6]{Huy16}).  Let $X_0$ be a complex $K3$ surface.  Then $X_0$ is compact K\"ahler by  Todorov and Siu.
Let
\[
        q_0:H^2(X_0,\mathbb Z)\times H^2(X_0,\mathbb Z)\to\mathbb Z
\]
be the cup-product intersection form. We denote by the same notation \(q_0\)
its \(\mathbb R\)-bilinear and \(\mathbb C\)-bilinear extensions to
\(H^2(X_0,\mathbb R)\) and \(H^2(X_0,\mathbb C)\), respectively. 
Let
\[
        D_{X_0}:=
        \{[\omega]\in\mathbb P(H^2(X_0,\mathbb C)):
        q_0(\omega,\omega)=0,\ q_0(\omega,\overline\omega)>0\}
\]
be the period domain associated with \(H^2(X_0,\mathbb Z)\) (\cite[Chapter 6, Section 1.1]{Huy16}). 

Note that $H^{2,0}(X_0)$ is $1$-dimensional and any generator of $H^{2,0}(X_0)$ is non-degenerate. Then one  easily obtains
the following  criterion (e.g., \cite[Chapter 6, Section 2.4]{Huy16}) which  will be used repeatedly: if \(X\) is a K3 surface,
\(0\neq\omega\in H^{2,0}(X)\), and \(0\neq\beta\in H^2(X,\mathbb R)\), then
\[
        \beta\in H^{1,1}(X,\mathbb R)
        \Longleftrightarrow
        q(\beta,\omega)=0.
\]

We now construct a period curve $\gamma$ as follows.
Take a  K\"ahler class  $\alpha$ on $X_0$.
 Choose any generator
$\omega_0\neq 0$ of $H^{2,0}(X_0)$.
Then 
\[
q_0(\alpha,\omega_0)=q_0(\alpha,\overline{\omega}_0)=0.
\]
Set
\[
a:=q_0(\alpha,\alpha)>0\text{  and  }Q:=q_0(\omega_0,\overline{\omega}_0)>0.
\]
For $s$ sufficiently small, define the non-zero class
\begin{equation}\label{def-omegas}
    \omega_s
:=
\omega_0+s\alpha-\frac{a}{2Q}s^2\overline{\omega}_0.
\end{equation}
Then
\[
q_0(\omega_s,\omega_s)=0\text{  and  }q_0(\omega_s,\overline{\omega_s})>0
\]
for $s$ sufficiently small. Then  any $[\omega_s]$ (for $s$ small enough) is in the period domain $D_{X_0}$.   Thus we have a holomorphic map
\[
\gamma:S\to D_{X_0},\qquad s\longmapsto[\omega_s],
\]
 for a certain small disk $S$ centered at $0$.

We now utilize  the  period curve $\gamma$ to produce a proper holomorphic submersion $\pi:X\to S$ via the local Torelli theorem.
Let
\[
\Pi:\mathfrak X\to \operatorname{Def}(X_0)
\]
be the smooth  universal deformation of $X_0$ (\cite[Chapter 6, Corollary 2.7]{Huy16}).  Recall the local Torelli theorem (\cite[Chapter 6, Proposition 2.8]{Huy16}) that the  period map
\[
P:\operatorname{Def}(X_0)\to D_{X_0}
\]
is a local isomorphism.  Note that $P(0)=\gamma(0)$. Then after shrinking both $\operatorname{Def}(X_0)$ and $S$ (still a disk centered at $0$), we 
have a well-defined morphism
\[
\iota:=P^{-1}\circ\gamma:S\to \operatorname{Def}(X_0).
\]

Set
\[
X:=\mathfrak X\times_{\operatorname{Def}(X_0)} S.
\]
As smoothness and properness are stable under base change, \[\pi: X \to S\] is a smooth proper morphism. Since $S$ is smooth,  $\pi$ is a proper holomorphic submersion between complex manifolds.  Clearly, each fiber $X_t:=\pi^{-1}(t)$ is a $K3$ surface.

\subsection{Failure of the K\"ahlerness of the total space  near the central K\"ahler fiber}

We will prove the following proposition.

\begin{proposition}\label{prop-k3}
Let \[
\pi: X\to S
\]  be the proper holomorphic submersion 
between complex manifolds constructed above, where $S\subset\mathbb C$ is a disk  centered at $0$.   Then there exists  a flat section
\[
\sigma\in H^0(S,R^2\pi_*\mathbb R),
\]
such that:
\begin{enumerate}[\rm{(}i\rm{)}]
\item $\sigma(0)$ (as in the convention fixed in Section \ref{sec-preli}) is a K\"ahler class on $X_0$;
\item $\sigma(s)$ is not of type $(1,1)$ on $X_s$ for each $0\neq s\in S$.
\end{enumerate}
Furthermore, for any smaller disk $0\in S'\subset S$, the total space
$X_{S'}:=\pi^{-1}(S')$ is not K\"ahler.
\end{proposition}

\begin{proof}
For each \(s\in S\), let
\[
        \tau_s:H^2(X_s,\mathbb C)\xrightarrow{\cong}H^2(X_0,\mathbb C)
\]
be the isomorphism induced by the Gauss--Manin parallel transport. Since \(S\) is a disk, the local system \(R^2\pi_*\mathbb C\) is constant and  these
isomorphisms globally trivialize \(R^2\pi_*\mathbb C\).

Furthermore, by Ehresmann’s theorem, the Gauss–Manin parallel transport is induced by differentiable identifications of the fibers. Hence it preserves the cup-product intersection pairing (e.g., \cite[Definition 3.1 and the following paragraph]{MM17}):
Denote by 
\[q_s: H^2\left(X_s, \mathbb{C}\right) \times H^2\left(X_s, \mathbb{C}\right) \rightarrow \mathbb{C}\] 
 the intersection form of the fiber $X_s$.
Then  for any $u, v \in H^2\left(X_s, \mathbb{C}\right)$, we have 
\begin{equation}\label{pre-inter}
    q_s(u, v)=q_0\left(\tau_s(u), \tau_s(v)\right).
\end{equation}

Denote by
\[
    P_{\pi}:S\to D_{X_0},  \text{   }  s\mapsto \left[\tau_s\left(H^{2,0}\left(X_s\right)\right)\right] 
\]
 the period map  (\cite[Chapter 6, Proposition 2.3]{Huy16}) of the family \(\pi: X\to S\), defined with
respect to the aforementioned Gauss--Manin trivialization. By the definitions of $\pi$ and  \(\iota\),
we have
\[
        P_{\pi}=P\circ\iota=\gamma.
\]
Consequently, 
\begin{equation}\label{taus-omegas}
     \tau_s\bigl(H^{2,0}(X_s)\bigr)=\mathbb C\omega_s.
\end{equation}

 Let
\[
\sigma\in H^0(S,R^2\pi_*\mathbb R)
\]
be the flat section corresponding to the fixed class $\alpha$. Then  we have
\[
        \tau_s(\sigma(s))=\alpha .
\]
Thus by \eqref{pre-inter} we obtain
\[
        q_s(\sigma(s),H^{2,0}(X_s))
        =
        q_0(\alpha,\mathbb C\omega_s).
\]
Since by \eqref{def-omegas}
\[
        q_0(\alpha,\omega_s)
        =
        q_0\!\left(\alpha,
        \omega_0+s\alpha-\frac{a}{2Q}s^2\overline{\omega}_0
        \right)
        =
        s\,q_0(\alpha,\alpha)
        =
        sa
        \neq 0
\]
for \(s\neq0\), the class \(\sigma(s)\) is not orthogonal to
\(H^{2,0}(X_s)\). Hence \(\sigma(s)\) is not of type \((1,1)\) on
\(X_s\) for $s\neq 0$. Clearly, $\sigma(0)=\alpha$ is of type \((1,1)\).

It remains to prove that the total space is not K\"ahler after any shrinking of the base.
Assume, to the contrary, that $ X_{S'}$ admits a K\"ahler form $\Omega$ for some
disk $0\in S'\subset S$. Then the fiberwise classes
\[
\beta(s):=[\Omega|_{X_s}]
\]
define a flat section of $R^2\pi_*\mathbb R|_{S'}$, and $\beta(s)$ is  of type
\((1,1)\) (K\"ahler class) on \(X_s\) for each $s\in S'$.   By \eqref{taus-omegas}
 we choose
\[
        \Omega_s:=\tau_s^{-1}(\omega_s)\in H^{2,0}(X_s).
\]
Since \(\beta(s)\in H^{1,1}(X_s,\mathbb R)\), we have
\[
        q_s(\beta(s),\Omega_s)=0.
\]
By  \eqref{pre-inter} and the fact  that
\(\tau_s(\beta(s))=\beta(0)\), we obtain
\[
        0
        =
        q_s(\beta(s),\Omega_s)
        =
        q_0(\beta(0),\omega_s)
\]
for each \(s\in S'\).
Then we obtain by \eqref{def-omegas} the polynomial identity
\[
q_0(\beta(0),\omega_0)
+s\,q_0(\beta(0),\alpha)
-\frac{a}{2Q}s^2q_0(\beta(0),\overline{\omega}_0)
=0.
\]
Hence
\[
q_0(\beta(0),\alpha)=0.
\]
This is impossible, because $\beta(0)$ and $\alpha$ are K\"ahler classes on 
$X_0$, and thus
\[
q_0(\beta(0),\alpha)=\int_{X_0}\beta(0)\wedge\alpha>0.
\]
This contradiction proves that $ X_{S'}$ is not K\"ahler, and the proof is complete.
\end{proof}

\begin{remark}\label{rem-product-k3}
The one-dimensional-base example above immediately gives higher-dimensional
examples by taking products with a polydisc. More precisely, for \(n\geq 2\),
let
\[
        B:=S\times \Delta^{n-1},
        \qquad
         Y:= X\times \Delta^{n-1},
\]
and let
\[
        \Pi: Y\to B
\]
be the product family. If \(p_1:B\to S\) denotes the first projection, then
\[
        \Sigma:=p_1^{-1}\sigma\in H^0(B,R^2\Pi_*\mathbb R)
\]
is a flat section. Under the natural identification
\[
        Y_{(s,u)}\cong X_s,
\]
the class \(\Sigma(s,u)\) is of type \((1,1)\) precisely when
\(\sigma(s)\) is of type \((1,1)\) on \(X_s\). 
In particular, the divisor $\{0\}\times \Delta^{n-1}\subset B$ consists entirely of points over which $\Sigma$ is of type $(1,1)$, whereas $\Sigma$ fails to be of type $(1,1)$ at all points with $s\neq 0$.
Clearly,  the total space is not K\"ahler after any shrinking of the base around $(0,0)$.
\end{remark}

\appendix

 \section{Leray spectral sequences}

For completeness and to fix the conventions and notations, we include proofs of several standard facts on Leray spectral sequences.

We first prove the following topological lifting lemma using the Leray spectral sequence.  Note that the statement of \cite[tag 01F4]{Stacks} gives the isomorphism 
\[
H^q(X,\mathcal F)\cong H^0(Y,R^qf_*\mathcal F)
\]
under the vanishing condition
\[
H^p(Y,R^qf_*\mathcal F)=0
\text{ for all }q\text{ and all }p>0.
\]
But it does not explicitly specify  the morphism which realizes this isomorphism.  
For applications, we give the following result.

\begin{lemma}\label{lem-leray-edge}
Let $f:X\to S$ be a continuous map, and let $\mathcal F$ be a sheaf of abelian groups on $X$.
Assume that
\[
H^p(S,R^qf_*\mathcal F)=0
\qquad\text{for each }p>0\text{ and each }q\geq 0.
\]
Then, for each $N\geq 0$, the natural  edge homomorphism
\[
\rho_N:H^N(X,\mathcal F)\to H^0(S,R^Nf_*\mathcal F)
\]
of the Leray spectral sequence
is an isomorphism.
\end{lemma}

\begin{proof}
Consider the Leray spectral sequence 
\[
E_2^{p,q}=H^p(S,R^qf_*\mathcal F)
\to
H^{p+q}(X,\mathcal F).
\]
By assumption, $E_2^{p,q}=0$ for all $p>0$ and all $q\geq 0$.
For fixed $N\geq 0$, the induced filtration on $H^N(X,\mathcal F)$ satisfies
\[
H^N(X,\mathcal F)=F^0H^N(X,\mathcal F)
\supset F^1H^N(X,\mathcal F)
\supset \cdots
\]
and
\[
F^pH^N(X,\mathcal F)/F^{p+1}H^N(X,\mathcal F)
\cong E_\infty^{p,N-p}.
\]
Since $E_\infty^{p,N-p}=0$ for all $p>0$, we have $F^1H^N(X,\mathcal F)=0.$
Furthermore,  $E_2^{0,N}=E_\infty^{0,N}$.

By definition, the $p=0$  edge homomorphism is the composite
\[
H^N(X,\mathcal F)
=
F^0H^N(X,\mathcal F)
\to
F^0H^N(X,\mathcal F)/F^1H^N(X,\mathcal F)
=
E_\infty^{0,N}
=
E_2^{0,N}.
\]
Since $F^1H^N(X,\mathcal F)=0$, this map is an isomorphism. Since
\[
E_2^{0,N}=H^0(S,R^Nf_*\mathcal F),
\]
the Leray edge homomorphism
\[
\rho_N:H^N(X,\mathcal F)\to H^0(S,R^Nf_*\mathcal F)
\]
is an isomorphism.
\end{proof}

We now establish the following compatibility between the edge homomorphism of the Leray spectral sequence for the constant sheaf and the restriction of cohomology classes.

\begin{lemma}\label{lem-H2-lift}
Let $\pi:\mathcal X\to U$ be a proper holomorphic submersion onto a polydisc
$U\subset \mathbb C^m$ centered at $0$.
For each $b\in U$, set $X_b:=\pi^{-1}(b)$.
Let
\[
\rho_{\mathbb R}:H^2(\mathcal X,\mathbb R)
\to H^0(U,R^2\pi_*\mathbb R)
\]
be the edge homomorphism of the Leray spectral sequence for the sheaf
$\mathbb R$ on $\mathcal X$.
Then, under the canonical identification
\[
(R^2\pi_*\mathbb R)_b\cong H^2(X_b,\mathbb R),
\]
the germ $\rho_{\mathbb R}(c)_b$ is identified with the restriction class
\[
c|_{X_b}\in H^2(X_b,\mathbb R).
\]
\end{lemma}

\begin{proof}
By definition, $R^2\pi_*\mathbb R$ is the sheaf associated with the presheaf
\[
\mathcal P^2:V\longmapsto \mathcal P^2(V):=H^2(\pi^{-1}(V),\mathbb R),
\]
where $V$ runs through the open subsets of $U$.

 Note that the canonical morphism
from the global sections of the  presheaf $\mathcal{P}^2$ to the global sections of its
associated sheafification $R^2\pi_*\mathbb R$  is exactly the edge homomorphism of the  Leray spectral
sequence
\[
E_2^{p,q}
=
H^p(S,R^qf_\ast\mathcal{H})
\Rightarrow
H^{p+q}(X,\mathcal{H})
\]
(\cite[\S 12.2.5]{EGAIII} or \cite[\S 5, p. 31]{AHK73}).
Then the edge
homomorphism  sends a class $c\in H^2(\mathcal X,\mathbb R)$ to the section
of $R^2\pi_*\mathbb R$ locally represented on $V$ by the restriction class
\[
c|_{\pi^{-1}(V)}\in H^2(\pi^{-1}(V),\mathbb R).
\]
Thus the germ $\rho_{\mathbb R}(c)_b$ is represented by the element
\[
c|_{\pi^{-1}(W)}\in H^2(\pi^{-1}(W),\mathbb R)
\]
for any sufficiently small open neighborhood $W\subset U$ of $b$.

Choose $W$ to be a  small polydisc centered at $b$. Since
$\pi$ is a proper holomorphic submersion, the restricted map
\[
\pi^{-1}(W)\to W
\]
is a $C^\infty$ locally trivial fiber bundle by Ehresmann's theorem. Since $W$ is contractible, the inclusion
\[
i_{b,W}:X_b\hookrightarrow \pi^{-1}(W)
\]
is a homotopy equivalence. Therefore the restriction map
\[
i_{b,W}^*:H^2(\pi^{-1}(W),\mathbb R)
\to H^2(X_b,\mathbb R)
\]
is an isomorphism.

The canonical identification
\[
(R^2\pi_*\mathbb R)_b\cong H^2(X_b,\mathbb R)
\]
is precisely the map induced on germs by these restriction maps
$i_{b,W}^*$. Hence the image of $\rho_{\mathbb R}(c)_b$ under this
identification is
\[
i_{b,W}^*\bigl(c|_{\pi^{-1}(W)}\bigr)
=
c|_{X_b}.
\]
This proves the assertion.
\end{proof}

For coherent sheaves, the compatibility is slightly more complicated. We now establish the following compatibility among the Leray edge homomorphism, the base-change map, and the restriction of cohomology classes.

\begin{lemma}
\label{lem-bc-edge}
Let $f:X\to S$ be a proper morphism of complex analytic spaces with analytic
fibers denoted by $X_t$, let $\mathcal F$ be a coherent
$\mathcal O_X$-module sheaf, let $t\in S$ be a point, and let $q\geq 0$ be an
integer. Let $ \mathfrak m_t\subset \mathcal O_{S,t}$
be the maximal ideal and $ k(t)=\mathcal O_{S,t}/\mathfrak m_t$ the residue field.
Denote by $ j_t:X_t\hookrightarrow X$
the natural closed embedding, and by
\[
        \mathcal F_t:=j_t^*\mathcal F 
\]
the restriction of $\mathcal F$.
Let
\[
        \operatorname{res}_t:
        H^q(X,\mathcal F)
        \to
        H^q(X,j_{t,*}\mathcal F_t)
        \cong
        H^q(X_t,\mathcal F_t)
\]
be the natural restriction map induced by the natural morphism  $ \mathcal F
        \to
        j_{t,*}\mathcal F_t$.

Let
\[
        \rho^q:
        H^q(X,\mathcal F)
        \to
        H^0(S,R^qf_*\mathcal F)
\]
be the edge homomorphism of the Leray spectral sequence associated with $f$ (Equivalently, $\rho^q$ is the morphism associated with the presheaf-to-sheafification morphism in the
definition of  $R^qf_*\mathcal F$).
 Let
\[
        \operatorname{ev}_t:
        H^0(S,R^qf_*\mathcal F)
        \to
        (R^qf_*\mathcal F)_t
        \otimes_{\mathcal O_{S,t}} k(t)
\]
be evaluation at $t$ (first taking the germ at $t$ and then passing to the
quotient). Let
\[
        \operatorname{bc}_t^q:
        (R^qf_*\mathcal F)_t
        \otimes_{\mathcal O_{S,t}} k(t)
        \to
        H^q(X_t,\mathcal F_t)
\]
be the base-change morphism obtained (by taking $M=k(t)$)
in the construction of
\cite[Chapter III, \S 3, p.  116]{BS76}. Then the diagram
\[
\begin{tikzcd}[column sep=large,row sep=large]
H^q(X,\mathcal F)
    \arrow[r,"{\rho^q}"]
    \arrow[dr,"{\operatorname{res}_t}" swap]
&
H^0(S,R^qf_*\mathcal F)
    \arrow[r,"{\operatorname{ev}_t}"]
&
(R^qf_*\mathcal F)_t
    \otimes_{\mathcal O_{S,t}} k(t)
    \arrow[dl,"{\operatorname{bc}_t^q}"]
\\
&
H^q(X_t,\mathcal F_t)
&
\end{tikzcd}
\]
is commutative.
\end{lemma}

\begin{proof}
We first explain and describe the restriction map $\operatorname{res}_t$ in a little more detail.
Let $I_t\subset \mathcal{O}_S$ be the coherent ideal sheaf defining the
reduced point $\{t\}\subset S$. Set
\[
\mathfrak{a}_t
:=
I_t\mathcal{O}_X
:=
\operatorname{Im}\bigl(
f^{-1}I_t\otimes_{f^{-1}\mathcal{O}_S}\mathcal{O}_X
\to
\mathcal{O}_X
\bigr).
\]
Then the analytic fiber $X_t$ is the closed complex subspace of $X$ defined
by $\mathfrak{a}_t$ (e.g., \cite[Chapter III, \S 3, p. 116]{BS76}), and
\[
j_{t,\ast}\mathcal{O}_{X_t}
\cong
\mathcal{O}_X/\mathfrak{a}_t.
\]
 Put
\[
G:=j_{t,\ast}\mathcal{F}_t .
\]
Since $\mathcal{F}_t=j_t^\ast\mathcal{F}$, we have canonical identifications
\[
G
\cong
\mathcal{F}\otimes_{\mathcal{O}_X}j_{t,\ast}\mathcal{O}_{X_t}
\cong
\mathcal{F}\otimes_{\mathcal{O}_X}(\mathcal{O}_X/\mathfrak{a}_t)
\cong
\mathcal{F}/\mathfrak{a}_t\mathcal{F}.
\]
Let
\[
\theta:\mathcal{F}\to G
\]
be the natural quotient morphism. Since $j_t$ is a closed embedding, the
canonical map
\[
H^q(X,j_{t,\ast}\mathcal{F}_t)
\to
H^q(X_t,\mathcal{F}_t)
\]
is an isomorphism (e.g., \cite[Chapter III, Lemma 2.10]{Ha77}). Under this identification, the morphism
\begin{equation}\label{formu-theta*}
    \theta_\ast:
H^q(X,\mathcal{F})
\to
H^q(X,G)
\cong
H^q(X_t,\mathcal{F}_t)
\end{equation}
is precisely the restriction map $\operatorname{res}_t$.

Recall the following standard description of the   edge
homomorphism of Leray spectral sequence. For an $\mathcal{O}_X$-module sheaf $\mathcal{H}$, the sheaf
$R^qf_\ast\mathcal{H}$ is the sheaf associated with the presheaf
$U\longmapsto H^q(f^{-1}(U),\mathcal{H})$
and the canonical morphism
\[
H^q(X,\mathcal{H})
\to
H^0(S,R^qf_\ast\mathcal{H})
\]
from the global sections of this presheaf to the global sections of its
associated sheafification sheaf is exactly the edge homomorphism of the  Leray spectral
sequence
\[
E_2^{p,q}
=
H^p(S,R^qf_\ast\mathcal{H})
\Rightarrow
H^{p+q}(X,\mathcal{H})
\]
(\cite[\S 12.2.5]{EGAIII} or \cite[\S 5, p. 31]{AHK73}).

Applying this description to the morphism of sheaves
\[
\theta:\mathcal{F}\to G
\]
and using the functoriality of sheafification (we may also directly use the theory of Grothendieck spectral sequence), we get the commutative diagram
\begin{equation}\label{formu-square}
    \begin{array}{ccc}
H^q(X,\mathcal{F})
&
\xrightarrow{\rho^q_{\mathcal{F}}}
&
H^0(S,R^qf_\ast\mathcal{F})
\\
\downarrow{\theta_\ast}
&&
\downarrow{\theta^q_S}
\\
H^q(X,G)
&
\xrightarrow{\rho^q_G}
&
H^0(S,R^qf_\ast G),
\end{array}
\end{equation}
where
\[
\theta^q_S
:=
H^0(S,R^qf_\ast\theta).
\]
Here $\rho^q_{\mathcal{F}}=\rho^q$, and $\rho^q_G$ denotes the corresponding
  edge homomorphism for $G$.

 Let \(U\subset S\) be an open neighbourhood of \(t\), and denote
by
\[
j_{t,U}:X_t\hookrightarrow f^{-1}(U)
\]
the induced closed embedding. Then
\[
G|_{f^{-1}(U)}
\cong
j_{t,U,\ast}\mathcal{F}_t .
\]
Consequently, the restriction morphism
\[
H^q(X,G)
\to
H^q(f^{-1}(U),G)
\]
is identified with the identity map of \(H^q(X_t,\mathcal{F}_t)\) through
the canonical identifications
\[
H^q(X,G)
=
H^q(X,j_{t,\ast}\mathcal{F}_t)
\cong
H^q(X_t,\mathcal{F}_t)
\]
and
\[
H^q(f^{-1}(U),G)
\cong
H^q(f^{-1}(U),j_{t,U,\ast}\mathcal{F}_t)
\cong
H^q(X_t,\mathcal{F}_t)
\]
(e.g., \cite[Chapter III, Lemma 2.10]{Ha77}).
It then follows that
\[
(R^qf_\ast G)_t
\cong
\varinjlim_{t\in U}H^q(f^{-1}(U),G)
\cong
H^q(X,G)
\cong
H^q(X_t,\mathcal{F}_t).
\]

We now use the aforementioned description of the Leray edge homomorphism as the
presheaf-to-sheafification map. For a class
\[
\xi\in H^q(X,G),
\]
the image of \(\rho^q_G(\xi)\) in the stalk \((R^qf_\ast G)_t\) is represented
by the compatible system of restrictions
\[
\xi|_{f^{-1}(U)}
\in
H^q(f^{-1}(U),G),
\qquad t\in U.
\]
Under the identifications above, all these restrictions correspond to the
same class \(\xi\in H^q(X,G)\). Therefore the composite
\[
H^q(X,G)
\xrightarrow{\rho^q_G}
H^0(S,R^qf_\ast G)
\to
(R^qf_\ast G)_t
\]
is the canonical identification
\[
H^q(X,G)
\cong
H^q(X_t,\mathcal{F}_t).
\]

Let
\[
\operatorname{ev}_{t,G}:
H^0(S,R^qf_\ast G)
\to
(R^qf_\ast G)_t
\cong
H^q(X_t,\mathcal{F}_t)
\]
denote the evaluation at $t$, followed by this canonical identification. Then we obtain
\begin{equation}\label{formu-eualiden}
    \operatorname{ev}_{t,G}\circ \rho^q_G
=
\operatorname{id}_{H^q(X,G)}
\end{equation}
under the canonical identification $H^q(X,G)\cong H^q(X_t,\mathcal{F}_t)$.

We now analyze the base-change
morphism. Recall the construction (\cite[Chapter III, \S 3, p. 116]{BS76}) of the base-change morphism (for \(M=k(t)\) in the notation of \cite[Chapter III, \S 3, p. 116]{BS76}),
 the morphism
\[
(R^qf_\ast\theta)_t:
(R^qf_\ast\mathcal{F})_t
\to
(R^qf_\ast G)_t
\]
induces the base-change morphism after passing to the fiber at \(t\). More
precisely, since
\[
G\cong \mathcal{F}/\mathfrak{a}_t\mathcal{F},
\qquad
\mathfrak{a}_t=I_t\mathcal{O}_X,
\]
the sheaf \(R^qf_\ast G\) is annihilated by \(I_t\). Thus
\[
\mathfrak m_t\cdot (R^qf_\ast G)_t=0,
\]
and the \(\mathcal{O}_{S,t}\)-linear morphism \((R^qf_\ast\theta)_t\)
factors uniquely through
\[
(R^qf_\ast\mathcal{F})_t
\to
(R^qf_\ast\mathcal{F})_t\otimes_{\mathcal{O}_{S,t}}k(t).
\]
The induced morphism
\[
(R^qf_\ast\mathcal{F})_t\otimes_{\mathcal{O}_{S,t}}k(t)
\to
(R^qf_\ast G)_t
\cong
H^q(X_t,\mathcal{F}_t)
\]
is precisely the base-change morphism \(bc^q_t\).

 Therefore
\[
bc^q_t\circ \operatorname{ev}_t
=
\operatorname{ev}_{t,G}\circ \theta^q_S .
\]

Combining this identity with the commutative diagram \eqref{formu-square}, we obtain
\[bc^q_t\circ \operatorname{ev}_t\circ \rho^q_{\mathcal{F}}
=
\operatorname{ev}_{t,G}\circ \theta^q_S\circ \rho^q_{\mathcal{F}}=
\operatorname{ev}_{t,G}\circ \rho^q_G\circ \theta_\ast.\]
By \eqref{formu-theta*} and \eqref{formu-eualiden}, this is exactly the $\operatorname{res}_t.$
This completes the proof.
\end{proof}

	\section*{Acknowledgement}
	
The author would like to express his sincere gratitude to Professors Sheng Rao and Quanting Zhao for their constant encouragement, generous support, and financial assistance.

\end{document}